\documentclass{article}

\usepackage[english]{babel}
\usepackage[a4paper]{geometry}
\usepackage{amssymb, amsmath, amsthm, mathtools}
\usepackage{graphicx, cite, color, enumitem}
\usepackage{longtable,pdflscape,booktabs,caption,multicol}
\usepackage{subcaption}

\usepackage[hyphens]{url}
\usepackage[colorlinks=true,citecolor=black,linkcolor=black,urlcolor=blue]{hyperref}
\usepackage[hyphenbreaks]{breakurl}

\theoremstyle{plain}
\newtheorem{theorem}{Theorem}
\newtheorem{lemma}[theorem]{Lemma}
\newtheorem{corollary}[theorem]{Corollary}
\newtheorem{proposition}[theorem]{Proposition}

\theoremstyle{definition}

\newtheorem{problem}[theorem]{Problem}

\usepackage{tikz,pgf}
\usetikzlibrary{calc}
\usepackage{float}

\DeclareMathOperator{\cart}{\square}
\DeclareMathOperator{\Circ}{Circ}
\DeclareMathOperator{\Aut}{Aut}
\DeclareMathOperator{\Dih}{Dih}
\DeclareMathOperator{\Sym}{Sym}

\newcommand{\arxiv}[2]{\href{https://arxiv.org/abs/#1}{\texttt{arXiv:#1}} \texttt{[#2]}}
\newcommand{\doi}[1]{\url{https://doi.org/#1}}

\usepackage{authblk}

\title{Nut graphs with a prescribed\\number of vertex and edge orbits}

\author[1,2,3]{Nino Bašić}
\author[1,4]{Ivan Damnjanović}

\affil[1]{FAMNIT, University of Primorska, Koper, Slovenia}
\affil[2]{IAM, University of Primorska, Koper, Slovenia}
\affil[3]{Institute of Mathematics, Physics and Mechanics, Ljubljana, Slovenia}
\affil[4]{Faculty of Electronic Engineering, University of Niš, Niš, Serbia}

\date{}

\begin{document}

\maketitle

\begin{abstract}
A nut graph is a nontrivial graph whose adjacency matrix has a one-dimensional null space spanned by a vector without zero entries. Recently, it was shown that a nut graph has more edge orbits than vertex orbits. It was also shown that for any even $r \ge 2$ and any $k \ge r + 1$, there exist infinitely many nut graphs with $r$ vertex orbits and $k$ edge orbits. Here, we extend this result by finding all the pairs $(r, k)$ for which there exists a nut graph with $r$ vertex orbits and $k$ edge orbits. In particular, we show that for any $k \ge 2$, there are infinitely many Cayley nut graphs with $k$ edge orbits and $k$ arc orbits.
\end{abstract}

\bigskip\noindent
{\bf Keywords:} nut graph, vertex orbit, edge orbit, arc orbit, Cayley graph, automorphism.

\bigskip\noindent
{\bf Mathematics Subject Classification:} 05C50, 05C25.

\section{Introduction}

A \emph{nut graph} is a nontrivial graph whose adjacency matrix has a one-dimensional null space spanned by a full vector, i.e., a vector without zero entries. The concept of nut graph was introduced by Sciriha and Gutman \cite{Sciriha1997, Sciriha1998_A, Sciriha1998_B, ScGu1998, Sciriha1999} and then investigated through a series of papers \cite{Sciriha2007, Sciriha2008, FoGaGoPiSc2020, GaPiSc2023}. The chemical justification for studying nut graphs can be found in \cite{ScFo2007, ScFo2008, FoPiToBoSc2014, CoFoGo2018, FoPiBa2021} and for more results on nut graphs, the reader is referred to the monograph \cite{ScFa2021}.

All the orders attainable by a $d$-regular nut graph were determined for each $d \le 4$ in \cite{GaPiSc2023}, and for each $d \in \{5, 6, \ldots, 11\}$ in \cite{FoGaGoPiSc2020}. Afterwards, a circulant graph based construction was used \cite{BaKnSk2022} to prove that there exists a $12$-regular nut graph of order $n$ if and only if $n \ge 16$. All the pairs $(n, d)$ for which there exists a $d$-regular circulant nut graph of order $n$ were subsequently determined through a series of papers \cite{DaSt2022, Damnjanovic2023_FIL, Damnjanovic2024_AMC}. Also, all the pairs $(n, d)$ with $4 \mid d$ for which there exists a $d$-regular Cayley nut graph of order $n$ were determined in \cite{Damnjanovic2023_ARX}. For more results on the realizability problems concerning nut graphs, see \cite{BaDam2024, BaDamFo2024, BaFo2024}. The complete classification of quartic circulant nut graphs was given in \cite{Damnjanovic2024_DMC}, while the complete classification of cubic tricirculant nut graphs and quartic bicirculant nut graphs was given in \cite{DaBaPiZi2024} and \cite{DaBaPiZi2025}, respectively. Moreover, it was shown that cubic tetra- and pentacirculant nut graphs do not exist \cite{BaDam2024}.

For a graph $G$, let $o_v(G)$, $o_e(G)$ and $o_a(G)$ denote the number of vertex orbits, edge orbits and arc orbits in $G$, respectively. The properties of vertex and edge orbits in nut graphs were recently investigated \cite{BaFoPi2024}, yielding the following two results.

\begin{theorem}[\hspace{1sp}{\cite[Theorem~2]{BaFoPi2024}}]\label{cool_th}
    Let $G$ be a nut graph. Then $o_e(G) \ge o_v(G) + 1$. 
\end{theorem}

\begin{theorem}[\hspace{1sp}{\cite[Theorem~34]{BaFoPi2024}}]\label{cool_th_2}
Let $r \geq 2$ be even. For every $k \ge r + 1$ there exist infinitely many nut graphs $G$ with $o_v(G) = r$ and $o_e(G) = k$.
\end{theorem}

Theorem \ref{cool_th} compares to the next two results by Buset.

\begin{theorem}[\hspace{1sp}{\cite[Theorem~1]{Buset1985}}]\label{buset_th_1}
    For any $r \in \mathbb{N}$ and $k \in \mathbb{N}_0$, there exists a graph $G$ with $o_v(G) = r$ and $o_e(G) = k$ if and only if $r \le 2k + 1$.
\end{theorem}

\begin{theorem}[\hspace{1sp}{\cite[Theorem~2]{Buset1985}}]\label{buset_th_2}
    For any $r \in \mathbb{N}$ and $k \in \mathbb{N}_0$, there exists a connected graph $G$ with $o_v(G) = r$ and $o_e(G) = k$ if and only if $r \le k + 1$.
\end{theorem}

With Theorem \ref{cool_th_2} in mind, this leads us to the natural problem of finding all the pairs $(o_v(G), o_e(G))$ attainable by a nut graph $G$. The main result of our paper is the complete resolution of this Buset-type problem. We first investigate the automorphisms of Cayley nut graphs and obtain the next result.

\begin{theorem}\label{cayley_th}
    For any $k \ge 2$, there exist infinitely many Cayley nut graphs $G$ with $o_e(G) = o_a(G) = k$.
\end{theorem}

We subsequently apply subdivisions to the Cayley nut graphs from Theorem \ref{cayley_th} to get the following theorem.

\begin{theorem}\label{main_th}
    For any $r \in \mathbb{N}$ and $k \ge r + 1$, there exist infinitely many nut graphs $G$ with $o_v(G) = r$ and $o_e(G) = k$.
\end{theorem}

In Section \ref{sc_prel}, we introduce the notation to be used in the next sections and preview some known results that we will need. Afterwards, we prove Theorem \ref{cayley_th} in Section \ref{sc_cayley} and Theorem \ref{main_th} in Section \ref{sc_main}. In Section \ref{sc_conclusion}, we end the paper with a brief conclusion.

\section{Preliminaries}\label{sc_prel}

All graphs considered will be undirected, simple and finite. We denote the standard $(0, 1)$-adjacency matrix of a graph $G$ by $A(G)$, and its spectrum, regarded as a multiset, by $\sigma(G)$. All spectral properties considered will correspond to the adjacency matrix. By $\Circ(n, S)$, where $S \subseteq \{1, 2, \ldots, \lfloor \frac{n}{2} \rfloor \}$, we denote the circulant graph on the vertex set $\mathbb{Z}_n$ such that any two vertices $u$ and $v$ are adjacent if and only if $u - v  \in S$ or $v - u  \in S$ (with the subtraction being done in $\mathbb{Z}_n$). We will need the following result on the spectra of circulant matrices.

\begin{lemma}[\hspace{1sp}{\cite[Section~3.1]{Gray2006}}]\label{circ_lemma}
    For any $n \in \mathbb{N}$, the eigenvalues of the circulant matrix
    \[
    C = \begin{bmatrix}
        c_0 & c_1 & c_2 & \cdots & c_{n-1}\\
        c_{n-1} & c_0 & c_1 & \cdots & c_{n-2}\\
        c_{n-2} & c_{n-1} & c_0 & \cdots & c_{n-3}\\
        \vdots & \vdots & \vdots & \ddots & \vdots\\
        c_1 & c_2 & c_3 & \dots & c_0
    \end{bmatrix}
    \]
    are of the form $P(\zeta)$, as $\zeta \in \mathbb{C}$ ranges over the $n$-th roots of unity, where
    \[
        P(x) = c_0 + c_1 x + c_2 x^2 + \cdots + c_{n-1} x^{n-1} .
    \]
    Moreover, for any $n$-th root of unity $\zeta$, the vector
    \[
        \begin{bmatrix} 1 & \zeta & \zeta^2 & \cdots & \zeta^{n-1} \end{bmatrix}^\intercal
    \]
    is an eigenvector of $C$ for the eigenvalue $P(\zeta)$.
\end{lemma}

We also need the next two results on nut graphs.

\begin{lemma}[\hspace{1sp}{\cite[Lemma~4.1]{ScGu1998}}]\label{subdiv_lemma}
    Let $G$ be a nut graph and let $G_1$ be the graph that arises from $G$ by subdividing an edge of $G$ four times. Then $G_1$ is also a nut graph.
\end{lemma}

\begin{lemma}[\hspace{1sp}{\cite[Theorem~2.4]{BaKnSk2022}}]\label{circ_prel}
    For any $k, n \in \mathbb{N}$ such that $k$ and $n$ are both even and $n \ge 2k + 2$, the graph $\Circ(n, \{1, 2, \ldots, k \})$ is a nut graph if and only if $\gcd(\frac{n}{2}, \frac{k}{2}) = \gcd(\frac{n}{2}, k + 1) = 1$.
\end{lemma}

We will use $G \cart H$ to denote the cartesian product \cite{ImKlaRa2008} of two graphs $G$ and $H$. A graph is \emph{prime} (with respect to the cartesian product) if it is not isomorphic to a graph of the form $G \cart H$, where $G$ and $H$ are both nontrivial. It is known that any connected graph $G$ has a unique factorization of the form $G = G_1 \cart G_2 \cart \cdots \cart G_k$, where $G_1, G_2, \ldots, G_k$ are all prime, up to isomorphisms and the order of the factors \cite[Theorem~15.1]{ImKlaRa2008}. This factorization is called the \emph{prime factorization} of $G$ and we say that two graphs $G$ and $H$ are coprime if they do not have a common factor in their respective prime factorizations. For a comprehensive treatment of various graph products, see \cite{HaImKla2011}.

\begin{lemma}[\hspace{1sp}{\cite[Section 1.4.6]{BrouHae2012}}]
\label{cart_lemma}
    For any two graphs $G$ and $H$, the spectrum of the cartesian product $G \cart H$ is given by
    \[
        \sigma(G \cart H) = \{ \lambda + \mu \colon \lambda \in \sigma(G), \, \mu \in \sigma(H) \} .
    \]
    Moreover, if $u \in \mathbb{R}^{V(G)}$ is an eigenvector of $G$ for the eigenvalue $\lambda$ and $v \in \mathbb{R}^{V(H)}$ is an eigenvector of $H$ for the eigenvalue $\mu$, then the vector $w \in \mathbb{R}^{V(G \cart H)}$ defined as
    \[
        w_{(g, h)} = u_g \, v_h \qquad (g \in V(G),\ h \in V(H))
    \]
    is an eigenvector of $G \cart H$ for the eigenvalue $\lambda + \mu$.
\end{lemma}

We denote the dihedral group of order $2n$ by $\Dih(n)$ and the symmetric group over $\mathbb{Z}_n$ by $\Sym(n)$. For any permutation group $\Gamma$ acting on $V$ and any point $x \in V$, the stabilizer of $x$ in $\Gamma$ will be denoted by $\Gamma_x$, and the orbit of $x$ in $\Gamma$ will be denoted by $x^\Gamma$.

\begin{lemma}[\hspace{1sp}{\cite[Lemma~2.2.2]{GodRoy2001}}]\label{orb_stab}
    Let $\Gamma$ be a permutation group acting on $V$ and let $x$ be a point in $V$. Then
    \[
        |\Gamma_x| \, |x^\Gamma| = |\Gamma| .
    \]
\end{lemma}

We use $\Aut G$ to denote the automorphism group of a graph $G$. We proceed with the following result on the automorphism groups of circulant graphs.

\begin{lemma}[\hspace{1sp}{\cite[Lemma~35]{BaFoPi2024}}]\label{dih_prel}
    For every $k \ge 1$ it holds that $\Aut(\Circ(n, \{1, 2, \ldots, k\})) \cong \Dih(n)$ for all $n \ge 2k + 3$.
\end{lemma}

We end the section by stating a few well-known facts about the cyclotomic polynomials. For any $n \in \mathbb{N}$, the \emph{cyclotomic polynomial} $\Phi_n(x)$ is defined as
\[
    \Phi_n(x) = \prod_{\xi}(x - \xi) ,
\]
where $\xi \in \mathbb{C}$ ranges over the primitive $n$-th roots of unity. It is known that all $\Phi_n(x)$ polynomials have integer coefficients and are irreducible in $\mathbb{Q}[x]$ (see, e.g., \cite[Chapter~33]{Gallian2017}). Therefore, any $\mathbb{Q}[x]$ polynomial contains a primitive $n$-th root of unity among its roots if and only if it is divisible by $\Phi_n(x)$.

\section{Edge and arc orbits of Cayley nut graphs}\label{sc_cayley}

In the present section, we prove Theorem \ref{cayley_th} through the following three propositions.

\begin{proposition}\label{cayley_prop_1}
For any even $k \ge 2$ and prime $p \ge k + 2$, the graph $\Circ(2p, \{1, 2, \ldots, k \})$ is a Cayley nut graph with $k$ edge orbits and $k$ arc orbits.
\end{proposition}
\begin{proof}
Let $G = \Circ(n, \{1, 2, \ldots, k \})$. The graph $G$ is obviously a Cayley graph. By Lemma~\ref{dih_prel}, we have $\Aut G \cong \Dih(2p)$, which means that for each $i \in \{1, 2, \ldots, k \}$, the set $\{ \{ v, v + i \} \colon v \in \mathbb{Z}_{2p} \}$ forms an edge orbit of $G$. From here, we obtain $o_e(G) = k$. Moreover, for any two adjacent vertices $v_1, v_2 \in \mathbb{Z}_n$, there is an automorphism of $G$ that swaps $v_1$ and $v_2$. Therefore, we also get $o_a(G) = k$. Finally, Lemma \ref{circ_prel} implies that $G$ is a nut graph.
\end{proof}

\begin{figure}[H]
\centering
\subcaptionbox{$\Circ(10, \{1, 2\})$}{
\begin{tikzpicture}[scale=1.0]
\tikzstyle{vertex}=[draw,circle,font=\scriptsize,minimum size=4pt,inner sep=1pt,fill=black]
\tikzstyle{edge}=[draw,thick]
\tikzstyle{medge}=[draw,thick,color=red]
\foreach \i in  {0,...,9} {
	\node[vertex] (v\i) at ({36*\i}:1.5) {};
}
\foreach \i in  {0,...,9} {
	\pgfmathtruncatemacro{\j}{mod(\i + 1, 10)}
	\pgfmathtruncatemacro{\k}{mod(\i + 2, 10)}
	\path[medge] (v\i) -- (v\k);
	\path[edge] (v\i) -- (v\j);
}
\end{tikzpicture}
}\qquad
\subcaptionbox{$\Circ(14, \{1, 2, 3, 4\})$}{
\begin{tikzpicture}[scale=1.0]
\tikzstyle{vertex}=[draw,circle,font=\scriptsize,minimum size=4pt,inner sep=1pt,fill=black]
\tikzstyle{edge}=[draw,thick]
\tikzstyle{medge}=[draw,thick,color=green!80!black]
\tikzstyle{aedge}=[draw,thick,color=red]
\tikzstyle{bedge}=[draw,thick,color=cyan]
\foreach \i in  {0,...,13} {
	\node[vertex] (v\i) at ({180/7*\i}:1.5) {};
}
\foreach \i in  {0,...,13} {
	\pgfmathtruncatemacro{\j}{mod(\i + 1, 14)}
	\pgfmathtruncatemacro{\k}{mod(\i + 2, 14)}
	\pgfmathtruncatemacro{\l}{mod(\i + 3, 14)}
	\pgfmathtruncatemacro{\m}{mod(\i + 4, 14)}
	\path[medge] (v\i) -- (v\k);
	\path[edge] (v\i) -- (v\j);
	\path[aedge] (v\i) -- (v\l);
	\path[bedge] (v\i) -- (v\m);
}
\end{tikzpicture}
}
\caption{Cayley nut graphs with two and four edge (arc) orbits from Proposition \ref{cayley_prop_1}. The edge (arc) orbits are color-coded.}
\end{figure}
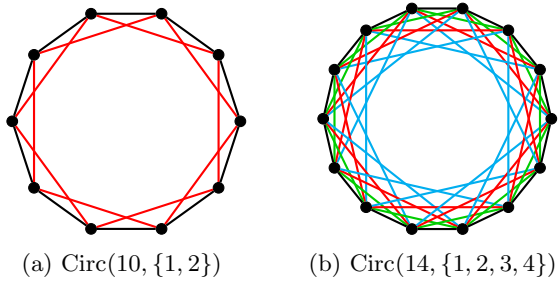

\begin{proposition}\label{cayley_prop_2}
For any odd $k \ge 5$ and prime $p \ge 2k + 1$, the graph $\Circ(2p, \{ 2, 3, \ldots, k - 1, \linebreak p \}) \cart K_2$ is a Cayley nut graph with $k$ edge orbits and $k$ arc orbits.
\end{proposition}
\begin{proof}
Let $G = \Circ(2p, \{ 2, 3, \ldots, k - 1, p \}) \cart K_2$ and let $\Gamma = \Aut G$. The graph $G$ is obviously a Cayley graph for the group $\mathbb{Z}_{2p} \times \mathbb{Z}_2$. By Lemma \ref{orb_stab}, we have $|\Gamma| = 4p \, |\Gamma_v|$, where $\Gamma_v$ is the stabilizer of some vertex $v$ in $\Gamma$. Now, let $\pi \in \Gamma_v$. Observe that
\[
    | N_G(v - (2, 0)) \setminus N_G(v) | = | N_G(v + (2, 0)) \setminus N_G(v) | = 6,
\]
while $|N_G(u) \setminus N_G(v) | \ge 7$ for any $u \in N_G(v) \setminus \{ v - (2, 0), v + (2, 0) \}$. Therefore, $\pi$ either fixes $v - (2, 0)$ and $v + (2, 0)$ or swaps them. Since there exists an automorphism from $\Gamma_v$ that swaps $v - (2, 0)$ and $v + (2, 0)$, Lemma \ref{orb_stab} implies that $|\Gamma_v| = 2 \, |\Gamma_{v, v + (2, 0)}|$, where $\Gamma_{v, v + (2, 0)}$ is the pointwise stabilizer of $v$ and $v + (2, 0)$ in $\Gamma$.

Now, let $\sigma \in \Gamma_{v, v + (2, 0)}$. Moreover, let
\[
    V_1 = \{ v + t(2, 0) \colon t \in \mathbb{Z} \} \qquad \mbox{and} \qquad V_2 = \{ v + (1, 0) + t(2, 0) \colon t \in \mathbb{Z} \}.
\]
Note that $\sigma$ must fix $v - (2, 0)$. By repeating the same argument from before, we conclude that $\sigma$ fixes all the vertices from $V_1$. Observe that any two distinct vertices from $V_2$ have distinct sets of neighbors among the vertices from $V_1$. Therefore, $\sigma$ also fixes all the vertices from $V_2$, and it is easy to see from here that $\sigma$ must be the identity permutation. Thus, we have $|\Gamma_{v, v + (2, 0)}| = 1$, which implies $|\Gamma| = 8p$.

Observe that $\Gamma$ contains the automorphisms $\varphi$ of the form
\[
    \varphi(v_1, v_2) = (f(v_1), v_2) \qquad \mbox{and} \qquad \varphi(v_1, v_2) = (f(v_1), v_2 + 1),
\]
where $v_1 \in \mathbb{Z}_{2p}$, $v_2 \in \mathbb{Z}_2$ and $f \in \Sym(2p)$ is a dihedral permutation. Since there are exactly $8p$ such automorphisms, this means that these are all the automorphisms of $\Gamma$. Therefore, for each $i \in \{ 2, 3, \ldots, k - 1, p \}$, the set $\{ \{ v, v + (i, 0) \} \colon v \in \mathbb{Z}_{2p} \times \mathbb{Z}_2 \}$ forms an edge orbit of $G$. In addition, the set $\{ \{ v, v + (0, 1) \} \colon v \in \mathbb{Z}_{2p} \times \mathbb{Z}_2 \}$ is another edge orbit of $G$. Thus, we have $o_e(G) = k$ and it also trivially follows that $o_a(G) = k$.

Note that $\sigma(K_2) = \{1, -1\}$, with both corresponding eigenvectors being full. By combining Lemmas \ref{cart_lemma} and \ref{circ_lemma}, we have that $\sigma(G)$ consists of the eigenvalues
\[
    \lambda(\zeta) = \sum_{j = 2}^{k - 1} \left( \zeta^j + \zeta^{-j} \right) + \zeta^p + 1 \qquad \mbox{and} \qquad \mu(\zeta) = \sum_{j = 2}^{k - 1} \left( \zeta^j + \zeta^{-j} \right) + \zeta^p - 1,
\]
as $\zeta$ ranges over the $2p$-th roots of unity. We trivially observe that
\[
    \lambda(1) = 2k - 2, \qquad \mu(1) = 2k - 4, \qquad \lambda(-1) = 2 \qquad \mbox{and} \qquad \mu(-1) = 0.
\]
Thus, to complete the proof, it suffices to show that $\lambda(\zeta) \neq 0$ and $\mu(\zeta) \neq 0$ for every $2p$-th root of unity $\zeta \neq 1, -1$.

By way of contradiction, suppose that $\lambda(\zeta) = 0$ or $\mu(\zeta) = 0$ for some $2p$-th root of unity $\zeta \neq 1, - 1$. Since $\zeta^p \in \{1, -1\}$, we have that at least one of the expressions
\[
    \sum_{j = 2}^{k - 1} \left( \zeta^j + \zeta^{-j} \right) + 2, \qquad \sum_{j = 2}^{k - 1} \left( \zeta^j + \zeta^{-j} \right) \qquad \mbox{and} \qquad \sum_{j = 2}^{k - 1} \left( \zeta^j + \zeta^{-j} \right) - 2
\]
equals zero. In other words, $\zeta$ is a root of at least one of the nonzero $\mathbb{Z}[x]$-polynomials
\[
    \sum_{j = 0}^{k - 3} x^j + \sum_{j = k + 1}^{2k - 2} x^j + 2x^{k - 1}, \qquad \sum_{j = 0}^{k - 3} x^j + \sum_{j = k + 1}^{2k - 2} x^j \qquad \mbox{and} \qquad \sum_{j = 0}^{k - 3} x^j + \sum_{j = k + 1}^{2k - 2} x^j - 2x^{k - 1},
\]
whose degrees are all $2k - 2$. This means that at least one of these three polynomials is divisible by $\Phi_p(x)$ or $\Phi_{2p}(x)$ because the order of $\zeta$ is either $p$ or $2p$. The contradiction now follows from $\deg \Phi_p = \deg \Phi_{2p} = p - 1 > 2k - 2$.
\end{proof}

\begin{proposition}\label{cayley_prop_3}
    For any odd $n \ge 5$, the graph $\Circ(2n, \{ 1, n \}) \cart K_4$ is a Cayley nut graph with three edge orbits and three arc orbits.
\end{proposition}
\begin{proof}
    Let $G' = \Circ(2n, \{1, n\})$ and $\Gamma' = \Aut G'$, as well as $G = G' \cart K_4$ and $\Gamma = \Aut G$. The graph $G$ is obviously a Cayley graph for the group $\mathbb{Z}_{2n} \times \mathbb{Z}_4$. Observe that $K_4$ is prime with respect to the cartesian product, while $G'$ and $K_4$ are coprime. For this reason, $\Gamma$ comprises the automorphisms $\varphi$ of the form
    \[
        \varphi(v_1, v_2) = (f_1(v_1), f_2(v_2)),
    \]
    where $v_1 \in \mathbb{Z}_{2n}$, $v_2 \in \mathbb{Z}_4$ and $f_1 \in \Gamma'$, while $f_2 \in \Sym(4)$ is any permutation (see, e.g., \cite[Theorem~15.5]{ImKlaRa2008}).

    Since $G'$ is vertex-transitive, Lemma \ref{orb_stab} gives $|\Gamma'| = 2n \, |\Gamma'_v|$, where $\Gamma'_v$ is the stabilizer of some vertex $v$ in $\Gamma'$. Now, let $\pi \in \Gamma'_v$. Observe that among the three neighbors $v - (1, 0)$, $v + (1, 0)$ and $v + (n, 0)$ of $v$, only $v - (1, 0)$ and $v + (1, 0)$ do not have an additional common neighbor apart from $v$. Indeed, $v - (1, 0)$ and $v + (n, 0)$ are both adjacent to $v + (n - 1, 0)$, while $v + (1, 0)$ and $v + (n, 0)$ are both adjacent to $v + (n + 1, 0)$. Therefore, $\pi$ either fixes $v - (1, 0)$ and $v + (1, 0)$ or swaps them. Since there exists an automorphism from $\Gamma'_v$ that swaps $v - (1, 0)$ and $v + (1, 0)$, Lemma \ref{orb_stab} implies that $|\Gamma'_v| = 2 \, |\Gamma'_{v, v + (1, 0)}|$, where $\Gamma'_{v, v + (1, 0)}$ is the pointwise stabilizer of $v$ and $v + (1, 0)$ in $\Gamma'$. By repeating the same argument, we trivially observe that $\Gamma'_{v, v + (1, 0)}$ only contains the identity permutation, hence $|\Gamma'| = 4n$ and $\Gamma' \cong \Dih(2n)$.
    
    Having characterized the group $\Gamma$, it follows that graph $G$ has exactly three edge orbits:\ $\{ \{v, v + (1, 0)\} \colon v \in \mathbb{Z}_{2n} \times \mathbb{Z}_4 \}$ and $\{ \{v, v + (n, 0)\} \colon v \in \mathbb{Z}_{2n} \times \mathbb{Z}_4 \}$, alongside $\{ \{v, v + (0, i)\} \colon v \in \mathbb{Z}_{2n} \times \mathbb{Z}_4, i \in \{1, 2, 3 \} \}$. Thus, we obtain $o_e(G) = 3$. Moreover, it is not difficult to see for any two adjacent vertices $v_1, v_2 \in V(G)$, there is an automorphism of $G$ that swaps $v_1$ and $v_2$. From here, we also get $o_a(G) = 3$.

    Note that $\sigma(K_4) = \{ 3, -1, -1, -1 \}$, with the eigenvector corresponding to $3$ being full. By combining Lemmas \ref{cart_lemma} and \ref{circ_lemma}, we conclude that $\sigma(G)$ consists of the eigenvalues
    \[
        \lambda(\zeta) = \zeta + \zeta^{-1} + \zeta^n + 3 \qquad \mbox{and} \qquad \mu_i(\zeta) = \zeta + \zeta^{-1} + \zeta^n - 1,
    \]
    as $\zeta$ ranges over the $2n$-th roots of unity and $i$ ranges over $\{1, 2, 3\}$. We trivially observe that
    \[
        \lambda(1) = 6, \qquad \mu_i(1) = 2, \qquad \lambda(-1) = 0 \qquad \mbox{and} \qquad \mu_i(-1) = -4.
    \]
    Thus, to complete the proof, it is enough to show that $\lambda(\zeta) \neq 0$ and $\mu_i(\zeta) \neq 0$ for every $2n$-th root of unity $\zeta \neq 1, - 1$.

    By way of contradiction, suppose that $\lambda(\zeta) = 0$ or $\mu_i(\zeta) = 0$ for some $2n$-th root of unity $\zeta \neq 1, - 1$. Since $\zeta^n \in \{1, -1\}$, we get $\zeta + \zeta^{-1} \in \{ -4, -2, 0, 2 \}$. The equality $\zeta + \zeta^{-1} = -4$ cannot hold because $|\zeta + \zeta^{-1}| \le |\zeta| + |\zeta^{-1}| = 2$. Moreover, $\zeta + \zeta^{-1} = -2$ and $\zeta + \zeta^{-1} = 2$ are equivalent to $\zeta = -1$ and $\zeta = 1$, respectively, hence we obtain a contradiction in both of these cases. Finally, $\zeta + \zeta^{-1} = 0$ cannot be satisfied because $n$ is assumed to be odd.
\end{proof}

\begin{figure}[H]
\centering
\begin{tikzpicture}[scale=1.0]
\tikzstyle{vertex}=[draw,circle,font=\scriptsize,minimum size=4pt,inner sep=1pt,fill=black]
\tikzstyle{edge}=[draw,thick]
\tikzstyle{medge}=[draw,thick,color=green!80!black]
\tikzstyle{aedge}=[draw,thick,color=red]
\tikzstyle{bedge}=[draw,thick,color=cyan]
\foreach \i in  {0,...,9} {
	\node[vertex] (a\i) at ({36*\i}:2) {};
	\node[vertex] (b\i) at ($ ({36*\i + 9}:2.8) $) {};
	\node[vertex] (c\i) at ($ ({36*\i + 18}:1.6) $) {};
	\node[vertex] (d\i) at ($ ({36*\i + 27}:2.4) $) {};
}
\foreach \i in  {0,...,4} {
    \pgfmathtruncatemacro{\k}{mod(\i + 5, 10)}
    \path[medge] (a\i) -- (a\k);
	\path[medge] (b\i) -- (b\k);
	\path[medge] (c\i) -- (c\k);
	\path[medge] (d\i) -- (d\k);
}
\foreach \i in  {0,...,9} {
	\pgfmathtruncatemacro{\j}{mod(\i + 1, 10)}
	\path[aedge] (a\i) -- (b\i);
	\path[aedge] (a\i) -- (c\i);
	\path[aedge] (a\i) -- (d\i);
	\path[aedge] (b\i) -- (c\i);
	\path[aedge] (b\i) -- (d\i);
	\path[aedge] (c\i) -- (d\i);
}
\foreach \i in  {0,...,9} {
     \pgfmathtruncatemacro{\j}{mod(\i + 1, 10)}
	\path[edge] (a\i) -- (a\j);
	\path[edge] (b\i) -- (b\j);
	\path[edge] (c\i) -- (c\j);
	\path[edge] (d\i) -- (d\j);
}
\end{tikzpicture}
\caption{The Cayley nut graph $\Circ(10, \{ 1, 5 \}) \cart K_4$ with three edge (arc) orbits from Proposition \ref{cayley_prop_3}. The edge (arc) orbits are color-coded.}
\end{figure}
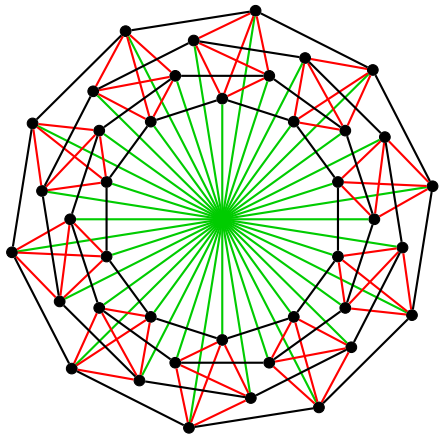

Theorem \ref{cayley_th} now follows from Propositions \ref{cayley_prop_1}--\ref{cayley_prop_3}.

\section{Vertex and edge orbits of nut graphs}\label{sc_main}

In this section, we start from Theorem \ref{cayley_th} and derive Theorem \ref{main_th}. We achieve this through the next lemma.

\begin{lemma}\label{construction_lemma}
    Let $G$ be a vertex-transitive nut graph such that $o_e(G) = o_a(G) = k$. Also, let $\mathcal{E}$ be an edge orbit of $G$ and let $G_1$ be the graph that arises from $G$ by subdividing each edge from $\mathcal{E}$ exactly $4t$ times, where $t \in \mathbb{N}$. Then $G_1$ is a nut graph such that $o_v(G_1) = 2t + 1$, $o_e(G_1) = 2t + k$ and $o_a(G_1) = 4t + k$.
\end{lemma}
\begin{proof}
    By repeated use of Lemma \ref{subdiv_lemma}, we conclude that $G_1$ is a nut graph. Let $V_1 = V(G)$ and $V_2 = V(G_1) \setminus V(G)$. The degree of $G$ is at least three, since $d$-regular nut graphs do not exist for $d < 3$ (see \cite{GaPiSc2023}). This means that the vertices from $V_1$ have a degree of at least three in $G_1$, while the vertices from $V_2$ all have the degree two in $G_1$. Therefore, for any $\pi \in \Aut(G_1)$, the restriction $\pi \restriction V_1$ is a permutation of $V_1$. For any two vertices $u, v \in V_1$ adjacent in $G$ and such that $uv \not\in \mathcal{E}$, we have $u \sim v$ in $G_1$, hence $\pi(u) \sim \pi(v)$ in $G_1$, which means that $\pi(u) \sim \pi(v)$ also holds in $G$. On the other hand, for any two vertices $u, v \in V_1$ adjacent in $G$ and such that $uv \in \mathcal{E}$, the graph $G_1$ contains a $(u, v)$-path of length $4t + 1$ whose internal vertices all have the degree two. This means that $G_1$ also has a $(\pi(u), \pi(v))$-path of length $4t + 1$ whose internal vertices all have the degree two, thus implying $\pi(u) \sim \pi(v)$ in $G$. With all of this in mind, we obtain that $\pi \restriction V_1 \in \Aut G$ holds for every $\pi \in \Aut G_1$.

    Since $G_1$ arises from $G$ by subdividing all the edges from the same edge orbit an equal number of times, it is not difficult to see that each automorphism of $G$ has a unique extension to $V(G_1)$ that is an automorphism of $G_1$. From here, we get $\Aut G_1 \cong \Aut G$. Also, since $o_e(G) = o_a(G)$, we know that for any edge $uv \in \mathcal{E}$, there exists an automorphism of $G$ that swaps $u$ and $v$. From here, we trivially observe that $o_v(G_1) = 2t + 1$, $o_e(G_1) = 2t + k$ and $o_a(G_1) = 4t + k$.
\end{proof}

We are now in a position to give the proof of Theorem \ref{main_th}.

\begin{proof}[Proof of Theorem~\ref{main_th}]
    If $r$ is even, then the result follows from Theorem \ref{cool_th_2}, while for $r = 1$, the result follows from Theorem \ref{cayley_th}. Now, assume that $r$ is odd and $r \ge 3$, and let $k \ge r + 1$. By Theorem \ref{cayley_th}, there exist infinitely many vertex-transitive nut graphs $G$ with $o_e(G) = o_a(G) = k - r + 1$. By applying the subdivision transformation from Lemma \ref{construction_lemma} with $t = \frac{r - 1}{2}$ to each of these graphs, we obtain infinitely many nut graphs $G_1$ with $o_v(G_1) = r$ and $o_e(G_1) = k$.
\end{proof}

\section{Conclusion}\label{sc_conclusion}

From Theorems \ref{cool_th} and \ref{cayley_th}, we obtain the following corollary.

\begin{corollary}\label{cayley_cor}
    For any $k \in \mathbb{N}_0$, the following holds:
    \begin{enumerate}[label=\textbf{(\alph*)}]
        \item there exists a Cayley nut graph with $k$ edge orbits if and only if $k \ge 2$;
        \item there exists a Cayley nut graph with $k$ arc orbits if and only if $k \ge 2$.
    \end{enumerate}
\end{corollary}

Corollary \ref{cayley_cor} solves the realizability problem for the number of edge orbits and arc orbits, respectively, among the Cayley nut graphs. As shown in Theorem \ref{cayley_th}, infinite realizability also holds in all the given cases. On the other hand, Theorems \ref{cool_th} and \ref{main_th} yield the next corollary.

\begin{corollary}\label{main_cor}
    For any $r \in \mathbb{N}$ and $k \in \mathbb{N}_0$, there exists a nut graph $G$ with $o_v(G) = r$ and $o_e(G) = k$ if and only if $k \ge r + 1$.
\end{corollary}

Corollary \ref{main_cor} solves the realizability problem for the pairs $(o_v, o_e)$ among the nut graphs. Theorem \ref{main_th} also implies that infinite realizability holds in all the given cases. The related realizability problem for the pairs $(o_v, o_a)$ still remains open.

\begin{problem}
    Determine all the pairs $(r, k)$ for which there exists a nut graph $G$ with $r$ vertex orbits and $k$ arc orbits.
\end{problem}

By Proposition \ref{cayley_prop_3}, there is a Cayley nut graph of order $n$ with three edge orbits for each $n \in \{40, 56, 72, 88, \ldots \}$. However, a computer search over the vertex-transitive graphs \cite{HoRoy2020, RoyHo2020} reveals three vertex-transitive nut graphs $G$ of order $16$ with $o_e(G) = o_a(G) = 3$, two of which are Cayley. Moreover, by Proposition \ref{cayley_prop_2}, there exists a Cayley nut graph of order $n$ with five edge orbits for $n \in \{ 44, 52, 68, 76, \ldots \}$. More precisely, there is such a Cayley nut graph for any $n = 4p$, where $p \ge 11$ is prime. However, there is a unique Cayley nut graph of order $12$ with $o_e(G) = o_a(G) = 5$. It is the Cayley graph for the group $\mathbb{Z}_6 \times \mathbb{Z}_2$ with the connection set $\{ (i, 0) \colon i \in \{1, 2, 4, 5 \}\} \cup \{ (i, 1) \colon i \in \{ 0, 1, 3, 5 \} \}$; see Figure~\ref{cool_graph}.

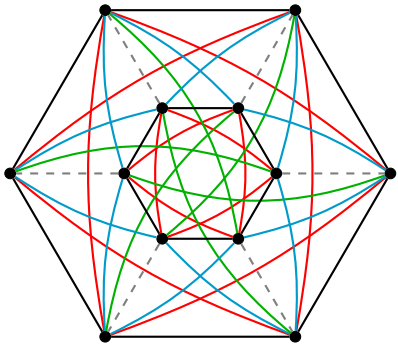
\begin{figure}[H]
\centering
\begin{tikzpicture}[scale=1.0]
\tikzstyle{vertex}=[draw,circle,font=\scriptsize,minimum size=4pt,inner sep=1pt,fill=black]
\tikzstyle{aedge}=[draw,thick,color=red]
\tikzstyle{edge}=[draw,thick]
\tikzstyle{medge}=[draw,thick,color=green!70!black]
\tikzstyle{bedge}=[draw,thick,color=cyan!80!black]
\tikzstyle{cedge}=[draw,thick,color=gray,dashed]
\foreach \i in  {0,...,5} {
	\node[vertex] (a\i) at ({60*\i}:1) {};
	\node[vertex] (b\i) at ($ ({60*\i}:2.5) $) {};
}
\foreach \i in  {0,...,5} {
\pgfmathtruncatemacro{\ka}{mod(\i + 2, 6)}
\pgfmathtruncatemacro{\kb}{mod(\i + 4, 6)}
\path[aedge] (a\i) to[bend right=10] (a\ka);
\path[aedge] (b\i) to[bend left=10] (b\kb);
}
\foreach \i in  {0,...,5} {
\pgfmathtruncatemacro{\j}{mod(\i + 3, 6)}
\pgfmathtruncatemacro{\ka}{mod(\i + 1, 6)}
\pgfmathtruncatemacro{\kb}{mod(\i + 5, 6)}
\path[medge] (b\i) to[bend left=20] (a\j);
\path[bedge] (a\i) to[bend right=10] (b\ka);
\path[bedge] (a\i) to[bend left=10] (b\kb);
}
\foreach \i in  {0,...,5} {
\pgfmathtruncatemacro{\k}{mod(\i + 1, 6)}
\path[cedge] (a\i) -- (b\i);
\path[edge] (a\i) -- (a\k);
\path[edge] (b\i) -- (b\k);
}
\end{tikzpicture}
\caption{The Cayley nut graph for the group $\mathbb{Z}_6 \times \mathbb{Z}_2$ with the connection set $\{ (i, 0) \colon i \in \{1, 2, 4, 5 \}\} \cup \{ (i, 1) \colon i \in \{ 0, 1, 3, 5 \} \}$. The edge (arc) orbits are color-coded.}
\label{cool_graph}
\end{figure}

\noindent
With this in mind, it makes sense to pose the following realizability problem.

\begin{problem}
    For any $k \ge 2$, determine all the orders for which:
    \begin{enumerate}[label=\textbf{(\alph*)}]
        \item there exists a Cayley nut graph with $k$ edge orbits;
        \item there exists a Cayley nut graph with $k$ arc orbits.
    \end{enumerate}
\end{problem}

\noindent
We conclude the paper with another natural problem.

\begin{problem}
    For any $r \in \mathbb{N}$ and $k \ge r + 1$, determine all the orders for which there exists a nut graph with $r$ vertex orbits and $k$ edge orbits.
\end{problem}

\section*{Acknowledgements}

N.\ Bašić is supported in part by the Slovenian Research Agency (research program P1-0294 and research project J5-4596). I.\ Damnjanović is supported by the Ministry of Science, Technological Development and Innovation of the Republic of Serbia, grant number 451-03-137/2025-03/200102, and the Science Fund of the Republic of Serbia, grant \#6767, Lazy walk counts and spectral radius of threshold graphs --- LZWK.

\end{document}